\documentclass[10pt]{article}
\usepackage{amsmath}
\usepackage{amssymb}
\usepackage{amsthm}
\usepackage{amsfonts}

\begin{document}

\def\D{\displaystyle}
\newtheorem{theorem}{Theorem}[section]
\newtheorem{definition}[theorem]{Definition}
\newtheorem{lemma}[theorem]{Lemma}
\newtheorem{algorithm}[theorem]{Algorithm}
\newtheorem{example}[theorem]{Example}
\newtheorem{proposition}[theorem]{Proposition}
\renewcommand{\theequation}{\thesection. \arabic{equation}}

\begin{center}
\large{\bf Computation of Bivariate Characteristic Polynomials of Finitely Generated Modules over Weyl Algebras}

\bigskip
\normalsize
Christian D\"onch
\smallskip

Research Institute for Symbolic Computation, Johannes Kepler University
A-4040 Linz, Austria
\medskip

Alexander Levin\\
\smallskip
The Catholic University of America\\
Washington, D. C. 20064
\end{center}

\bigskip

\begin{abstract}In this paper we generalize the classical Gr\"obner basis technique to prove the existence and present a method of computation of a characteristic polynomial in two variables associated with a finitely generated module over a Weyl algebra. We also present corresponding algorithms and examples of the computation of such polynomials, which, in particular, illustrate the fact that a bivariate characteristic polynomial can contain some invariants that are not carried by the Bernstein dimension polynomial. We also obtain a generalization of our results on to the case of multivariate characteristic polynomials associated with arbitrary partition of the basic sets of indeterminates and derivations of a Weyl algebra.
\end{abstract}

\bigskip

\normalsize

\section{Introduction}

\medskip

The role of Hilbert polynomials in commutative algebra and algebraic geometry
is well known. In [2] I. Bernstein introduced an analog of Hilbert polynomial
for a finitely generated filtered module over a Weyl algebra and extended
the theory of multiplicity to the class of such modules. The results of this
study have found interesting analytical applications (many of them
are considered in  Bj\"ork's book [4]). In particular, they allowed
I. Bernstein [3] to prove the Gelfand's conjecture on meromorphic
extensions of functions $\Gamma_{f} (\lambda) = \int P^{\lambda}(x)f(x)dx$
of one complex variable $\lambda$ defined in the half-space
$Re(\lambda) > 0$ for any polynomial in $n$ real variables
$P(x) = P(x_{1},\dots, x_{n})$  and for any function
$f(x) = f(x_{1},\dots, x_{n})\in C_{0}^{\infty}({\bf R}^{n})$.

In what follows we prove the existence, determine invariants and outline
methods of computation of dimension polynomials in two variables
associated with the natural bifiltration of a finitely generated module
over a Weyl algebra $A_{n}(K)$. We show that
such polynomials not only characterize the Bernstein class of
left $A_{n}(K)$-modules, but also carry, in general, more
invariants than dimension polynomials introduced by Bernstein.

\section{Preliminaries}

Throughout the paper {\bf Z}, {\bf N} and {\bf Q} denote the sets of all integers, all non-negative integers and all rational numbers, respectively. As usual, ${\bf Q}[t]$ denotes the ring of polynomials in one variable $t$ with rational coefficients and $o(t^{n})$ denotes a polynomial from ${\bf Q}[t]$ of degree less than $n$. By a ring we always mean an associative ring with a unit. Every ring homomorphism is unitary (maps unit onto unit), every subring of a ring contains the unit of the ring. Unless otherwise indicated, by the module over a ring $R$ we mean a unitary left $R$-module.

In what follows we consider a Weyl algebra as an algebra of differential operators over a polynomial ring. More precisely, let $K$ be a field of
zero characteristic and $R = K[x_{1},\dots, x_{n}]$ a polynomial ring in $n$ variables $x_{1},\dots, x_{n}$ over $K$. Furthermore, let $\partial_{i}$ denote the operator of partial differentiation of the ring $R$ with respect to the variable $x_{i}$ ($i=1,\dots, n$) and let $A_{n}(K)$ denote the corresponding ring of differential operators over $R$. Then $A_{n}(K)$ is said to be a Weyl algebra in $n$ variables with coefficients from $K$.
It is clear that the $K$-algebra $A_{n}(K)$ is generated by the elements $x_{1},\dots, x_{n}, \partial_{1}, \dots, \partial_{n}$, $\partial_{i}\partial_{j} = \partial_{i}\partial_{j}$ and $\partial_{i}x_{j} = x_{j}\partial_{i}$ for any two different indices $i$ and $j$ ($1\leq i, j\leq n$), and $\partial_{i}x_{i} = x_{i}\partial_{i} + 1$ for $i=1,\dots, n$. (The last identity is a consequence of the product rule, if one considers actions of the operators $\partial_{i}x_{i}$ and $x_{i}\partial_{i}$ on the ring $R$: $(\partial_{i}x_{i})(P) = \partial_{i}(x_{i}P) =
(x_{i}\partial_{i})(P) + P$ for any $P\in R$.)

In what follows, multi-indices with non-negative integers are denoted by small Greek letters. Thus, monomials $x_{1}^{\alpha_{1}}\dots x_{n}^{\alpha_{n}}$ and $\partial_{1}^{\beta_{1}}\dots \partial_{n}^{\beta_{n}}$ are written as $x^{\alpha}$ and $\partial^{\beta}$, their total degrees $\alpha_{1}+\dots +\alpha_{n}$ and $\beta_{1}+\dots +\beta_{n}$ are denoted by $|\alpha|$ and $|\beta|$, respectively.

It is known (see [4, Chapter 1, Proposition 1.2]) that monomials $x^{\alpha}\partial^{\beta}$ ($\alpha, \beta \in {\bf N}^{n}$) form a basis
of $A_{n}(K)$ over the field $K$, so that every element $D\in A_{n}(K)$ can be written in a unique way as a finite sum $\D\sum_{\alpha, \beta} k_{\alpha \beta}x^{\alpha}\partial^{\beta}$ with the coefficients

\smallskip

\noindent$k_{\alpha \beta}\in K$. The number $ord\, D = \max \{|\alpha|+|\beta| | k_{\alpha \beta}\neq 0\}$ is called the order of the element $D$.

Since $ord\,(D_{1}D_{2}) = ord\,D_{1} + ord\,D_{2}$ for any $D_{1}, D_{2}\in A_{n}(K)$, the Weyl algebra $A_{n}(K)$ can be considered as a filtered ring with the nondecreasing filtration $(W_{r})_{r\in {\bf Z}}$ where $W_{r} = \{D\in A_{n}(K) | ord\,D\leq r\}$ for $r\in {\bf N}$ and $W_{r} = 0$, if $r<0$.

If $M$ is a finitely generated left $A_{n}(K)$-module with a system of generators $g_{1},\dots, g_{p}$, then $M$ can be naturally considered as
a filtered $A_{n}(K)$-module with the filtration $(M_{r})_{r\in {\bf Z}}$ where  $M_{r} = \sum_{i=1}^{p}W_{r}g_{i}$ for $r \in {\bf Z}$. It is clear
that each $M_{r}$ is a finitely generated vector $K$-space, $W_{r}M_{s} = M_{r+s}$ for all $r, s\in {\bf N}$, and $\bigcup_{r\in {\bf N}}M_{r} = M$.

The following statement is proved in [2] (see also {[4, Chapter 1, Corollaries 3.3, 3.5, and Theorem 4.1]).

\medskip

\begin{proposition}
With the above notation, there exists a polynomial $\psi_{M}(t)\in {\bf Q}[t]$ with the following properties.

{\em (i)}\, $\psi_{M}(r) = dim_{K}M_{r}$ for all sufficiently large $r\in {\bf Z}$ (i. e., there exists $r_{0}\in {\bf Z}$ such that the last equality holds
for all integers $r\geq r_{0}$);

{\em (ii)}\, $n\leq deg\,\psi(t)\leq 2n$;

{\em (iii)}\, If $\psi(t) = a_{d}t^{d} +\dots + a_{1}t + a_{0}$ ($a_{d},\dots, a_{1}, a_{0}\in {\bf Q}$), then the degree $d$ of the polynomial $\psi(t)$ and the integer $d!a_{d}$ do not depend on the choice of the system of generators $g_{1},\dots, g_{p}$ of $M$. These numbers are denoted by $d(M)$ and $e(M)$, they are called the Bernstein dimension and multiplicity of the module $M$, respectively. \hspace{4in}\framebox[0.05in][l]{}
\end{proposition}

\medskip

The polynomial $\psi_{M}(t)$ is called the {\em Bernstein polynomial\/} of the $A_{n}(K)$-module $M$ associated with the given system of generators.
The family of all finitely generated left $A_{n}(K)$-modules $M$ such that $d(M) = n$ is denoted by $\mathcal{B}_{n}$, it is called the
{\em Bernstein class\/} of $A_{n}(K)$-modules.

The following statement (see [4, Chapter 1, Proposition 5.3 and Theorem 5.4]) gives some properties of the Bernstein class.

\begin{proposition}

{\em (i)}\, If $0\rightarrow M_{1}\rightarrow M_{2}\rightarrow M_{3}\rightarrow 0$
is an exact sequence of left $A_{n}(K)$-modules, then
$M_{2}\in \mathcal{B}_{n}$ if and only if $M_{1}\in \mathcal{B}_{n}$ and
$M_{3}\in \mathcal{B}_{n}$.

{\em (ii)}\, If $M\in \mathcal{B}_{n}$, then $M$ has a finite length as a left
$A_{n}(K)$-module. In fact, every strictly increasing sequence of
$A_{n}(K)$-modules contains at most $e(M)$ terms.

{\em (iii)}\, If $M$ is any filtered $A_{n}(K)$-module with an increasing
filtration $(M_{r})_{r\in {\bf Z}}$ and there exist positive integers
$a$ and $b$ such that $dim_{K}M_{r}\leq ar^{n} + b(r+1)^{n-1}$ for all
$r\in {\bf N}$, then $M\in \mathcal{B}_{n}$ and $e(M)\leq n!a$.
\hspace{2in}\framebox[0.05in][l]{}
\end{proposition}

\section {Numerical polynomials in two variables}

\begin{definition}
A polynomial $f(t_{1}, t_{2})$ in two variables $t_{1}$ and $t_{2}$ with
rational coefficients is called numerical if $f(t_{1}, t_{2})\in {\bf Z}$
for all sufficiently large $t_{1}, t_{2}\in{\bf Z}$, i.e., there exists
an element $(r_{0}, s_{0})\in {\bf Z}^2$ such that  $f(r, s)\in {\bf Z}$
for all integers $r\geq r_{0},\, s\geq s_{0}$.
\end{definition}

\medskip

It is clear that every polynomial in two variables with integer coefficients
is numerical.  As an example of a numerical polynomial in two variables with
noninteger coefficients one can consider a polynomial
$\D{t_{1}\choose m}{t_{2}\choose n}$, where $m$ and $n$ are positive integers
at least one of which is greater than 1. (As usual, for any $k\in {\bf Z},
k\geq 1$, $\D{t\choose k}$ denotes the polynomial
$\D{t\choose k} =\D\frac{t(t-1)\dots (t-k+1)}{k!}$ in one variable $t$;
furthermore, we set
$\D{t\choose0} = 1$, and $\D{t\choose k} = 0$ if $k$ is a negative integer).

By the degree of a monomial $u=t_{1}^{i}t_{2}^{j}$
we mean its total degree $deg\, u = i + j$, and the degrees of $u$
relative to $t_{1}$ and $t_{2}$ are defined as $deg_{t_{1}}\, u = i$ and
$deg_{t_{2}}\, u = j$, respectively.
If $f(t_{1}, t_{2}) = a_{1}u_{1} + \dots + a_{k}u_{k}$ is a
representation of a numerical polynomial $f(t_{1}, t_{2})$ as a sum of
monomials $u_{1}, \dots, u_{k}$ with nonzero coefficients
$a_{1}, \dots, a_{k}$, then the degree of $f(t_{1}, t_{2})$ and the degree
of this polynomial relative to $t_{i}$ ($i = 1, 2$) are defined as
usual: $deg\, f = \max \{deg\, u_{i}|1\leq i\leq k\}$ and
$deg_{t_{i}}\, f = \max \{deg_{t_{i}}\, u_{i}|1\leq i\leq k\}$,
respectively.

The following proposition proved in [9] gives a ''canonical'' representation
of a numerical polynomial in two variables.

\begin{proposition}
Let $f(t_{1}, t_{2})$ be a numerical polynomial
in two variables $t_{1}$, $t_{2}$, and let $deg_{t_{1}}\, f = p$,
$deg_{t_{2}}\, f = q$ . Then the
polynomial $f(t_{1}, t_{2})$ can be represented in the form
$$\hspace{1in}f(t_{1}, t_{2}) = \sum_{i=0}^{p}\sum_{j=0}^{q}
{a_{ij}{t_{1}+i\choose i}{t_{2}+j\choose j}}\hspace{1in}(3.1)$$
\noindent with integer coefficients $a_{ij}$
($0\leq i\leq p,\, 0\leq j\leq q$)
that are uniquely defined by the polynomial  $f(t_{1}, t_{2})$.
\end{proposition}

In what follows (until the end of the section), we deal with subsets of
the set ${\bf N}^{m+n}$ where $m$ and $n$ are
positive integers. If $A\subseteq{\bf N}^{m+n}$, then $A(r,s)$
($r, s\in {\bf N}$) will denote the subset of $A$ that consists of all
$(m+n)$-tuples $(a_{1},\dots , a_{m+n})$ such that
$a_{1}+ \dots +a_{m}\leq r$ and $a_{m+1}+ \dots +a_{m+n}\leq s$.
Furthermore, $V_{A}$ will denote the set
$\{v=(v_{1}, \dots , v_{m+n})\in {\bf N}^{m+n}|v$ is not greater than or
equal to any element of $A$ with respect to the product order on
${\bf N}^{m+n}\}$. (Recall that the product order on the set ${\bf N}^{k}$
($k\in {\bf N}, k\geq 1$) is a partial order $\leq_P$ on ${\bf N}^{k}$ such
that $(c_{1}, \dots , c_{k})\leq_P (c'_{1}, \dots , c'_{k})$ if and only if
$c_{i}\leq c'_{i}$ for all $i=1, \dots , k$.) Clearly, an element
$v=(v_{1}, \dots , v_{m+n})\in {\bf N}^{m+n}$ belongs to $V_{A}$ if
and only if for any element  $(a_{1},\dots , a_{m+n})\in A$
there exists $i\in {\bf N}, \, 1\leq i\leq m+n$, such that $a_{i} > v_{i}$.

The following two statements proved in
[10, Chapter II, Theorem 2.2.5 and Proposition 2.2.11]
generalize the well-known
Kolchin's result on numerical polynomials associated with subsets of
${\bf N}$ (see [8, Chapter 0, Lemma 17]) and give the explicit
formula for the numerical polynomials in two variables associated with a
finite subset of ${\bf N}^{m+n}$ ($m$ and $n$ are fixed positive integers).

\medskip

\begin{proposition}
With the above notation, for any set $A\subseteq{\bf N}^{m+n}$, there exists
a numerical polynomial $\omega_{A}(t_{1}, t_{2})$ in two
variables $t_{1}, t_{p}$ such that

\smallskip

{\em (i)}\, $\omega_{A}(r, s) = Card\,V_{A}(r,s)$ for
all sufficiently large $r, s\in {\bf N}$
(as usual, $Card \, V$ denotes the number of elements of a finite set $V$);

\smallskip

{\em (ii)}\, $deg\,\omega \leq m+n$,  $deg_{t_{1}}\omega \leq m$, and
$deg_{t_{2}}\omega \leq n$;

\smallskip

{\em (iii)}\, $deg\,\omega = m + n$ if and only if the set $A$ is empty,
in this case

$\omega_{A}(t_{1}, t_{2}) =
\D{{t_{1}+m}\choose m}{{t_{2}+n}\choose n}$;

\smallskip

{\em (iv)}\, $\omega_{A}(t_{1}, t_{2}) = 0$ if and only if
$(0, 0)\in A$.
\end{proposition}

\begin{definition}
The polynomial $\omega_{A}(t_{1},\dots, t_{p})$, whose existence is
established by Proposition {\em 3.3}, is called the $(m, n)$-dimension
polynomial of the set $A\subseteq {\bf N}^{m+n}$.
\end{definition}
\medskip

\begin{proposition}
Let $A = \{a_{1}, \dots, a_{p}\}$ be a finite subset of ${\bf N}^{m+n}$
($m$ and $n$ are fixed positive integers) and let
$a_{i} = (a_{i1}, \dots, a_{i,m+n})$ \, for $i=1,\dots, p$.
Furthermore, for any $l\in {\bf N}$, $0\leq l\leq p$, let $\Gamma (l,p)$
denote the set of all $l$-element subsets of the set
${\bf N}_{p} = \{1,\dots, p\}$, and for any $\sigma \in \Gamma (l,p)$
let $\bar{a}_{\sigma j} = \max \{a_{ij} | i\in \sigma\}$ ($1\leq j\leq m+n$),
$b_{\sigma} = \sum_{j=1}^{m}\bar{a}_{\sigma j}$, and
$c_{\sigma} = \sum_{j=m+1}^{m+n}\bar{a}_{\sigma j}$.
Then
$$\omega_{A}(t_{1},t_{2}) =
\sum_{l=0}^{p}(-1)^{l}\sum_{\sigma \in \Gamma (l,p)}
{t_{1}+m - b_{\sigma}\choose m}{t_{2} +n - c_{\sigma}\choose n}.$$
\end{proposition}

Let ${\bf N}_{p}=\{1,\dots,p\}$ ($p\in {\bf Z}, p\geq 1$) be the
set of the first $p$ positive integers and let ${\bf N}^{n}\times{\bf N}_{p}$
be the cartesian product of $n$ copies of ${\bf N}$ ($n\in {\bf N}$) and
${\bf N}_{p}$ considered as an ordered set with respect to the
product order $\leq _{P}$ such that
$(a_{1},\dots, a_{n}, b)\leq _{P} (a'_{1},\dots, a'_{n}, b')$ if and only if
$a_{i}\leq a'_{i}$ for all $i=1,\dots, n$ and $b\leq b'$. As usual, if
$(a_{1},\dots, a_{n}, b)\leq _{P} (a'_{1},\dots, a'_{n}, b')$ and
$(a_{1},\dots, a_{m}, b)\neq (a'_{1},\dots, a'_{n}, b')$, we write
$(a_{1},\dots, a_{n}, b)<_{P} (a'_{1},\dots, a'_{n}, b')$.

In what follows, we will need the following result on the order $\leq _{P}$
whose proof can be found in [8, Chapter 0, Sect. 17]

\begin{lemma}
Every infinite sequence of elements of ${\bf N}^{n}\times{\bf N}_{p}$
($n, p\in {\bf N}$, $p\geq 1$) has an infinite subsequence, strictly
increasing relative to the product order,
in which every element has the same projection on ${\bf N}_{p}$.
\hspace{1.2in}\framebox[0.05in][l]{}
\end{lemma}

\section{Reduction in a free $A_{n}(K)$-module. $(x, \partial)$-Gr\"obner bases}

The efficiency of the classical Gr\"obner basis methods for the
computation of Hilbert polynomials of graded and filtered modules over
polynomial rings is well-known. (One of the best presentations of the
appropriate results and algorithms can be found in [1, Chapter 9]
and [6, Section 15.10].)  Similarly, the generalization of the Gr\"obner
basis technique to the rings of differential operators developed in
[7] and  [10, Chapter 4] allows to find dimension polynomials of
finitely generated differential modules
(see [10, Chapter 4, theorem 4.3.5]). In this section we will generalize
the classical Gr\"obner reduction to the case when the
set of terms of a Weyl algebra $A_{n}(K)$ is considered together with two
natural orderings. The results obtained allows to prove the existence
and give a method of computation of characteristic polynomials
in two variables associated with a finite system of generators of an
$A_{n}(K)$-module.

\medskip

In what follows, we keep the notation and conventions of Section 2.
In particular, $A_{n}(K)$ denotes a Weyl algebra in $n$ variables
$x_{1},\dots, x_{n}$ over a field $K$ of zero
characteristic, and the appropriate partial differentiations
are denoted by $\partial_{1},\dots, \partial_{n}$, respectively.
Furthermore, $\Theta$ will denote the set of all power
products of the form $x_{1}^{\alpha_{1}}\dots x_{n}^{\alpha_{n}}\partial_{1}^{\beta_{1}}
\dots \partial_{n}^{\beta_{n}}$ with nonnegative integer exponents; such a power product will be
called a {\em monomial} and denoted by $x^{\alpha}\partial^{\beta}$.  (We use the standard notation
with multi-indices: unless otherwise is indicated, $\alpha$, $\beta$, etc. denote multi-indices
$(\alpha_{1},\dots, \alpha_{n})$, $(\beta_{1},\dots, \beta_{n})$, etc., $x^{\alpha}$ and $\partial^{\beta}$ denote
the monomials $x_{1}^{\alpha_{1}}\dots x_{n}^{\alpha_{n}}$ and $\partial_{1}^{\beta_{1}}
\dots \partial_{n}^{\beta_{n}}$, respectively, $x^{\alpha + \gamma} = x_{1}^{\alpha_{1}+\gamma_{1}}\dots x_{n}^{\alpha_{n}+\gamma_{n}}$,
$\partial^{\beta + \sigma} = \partial_{1}^{\beta_{1}+\sigma_{1}}\dots \partial_{n}^{\beta_{n}+\sigma_{n}}$, etc.)

If $\theta = x^{\alpha}\partial^{\beta}$, then the power products
$x^{\alpha} = x_{1}^{\alpha_{1}}\dots x_{n}^{\alpha_{n}}$ and
$\partial^{\beta} = \partial_{1}^{\beta_{1}}\dots \partial_{n}^{\beta_{n}}$ will be denoted
by $\theta_{x}$ and $\theta_{\partial}$, respectively. It is easy to see that the
sets $\Theta_{x} = \{\theta_{x}\,|\theta\in\Theta\}$ and
$\Theta_{\partial} = \{\theta_{\partial}\,|\theta\in\Theta\}$ are commutative multiplicative semigroups
(of course, $\Theta$ is not: $\partial_{i}x_{i}\neq x_{i}\partial_{i}$ for $i=1,\dots, n$).

\medskip

For any element $\theta = x^{\alpha}\partial^{\beta}\in \Theta$, the numbers
$|\alpha|$ and $|\beta|$ will be called, respectively, the $x$-order and $\partial$-order of $\theta$ or the  the orders of $\theta$ relative
to the sets $\{x_{1},\dots, x_{n}\}$ and $\{\partial_{1},\dots, \partial_{n}\}$, respectively. These numbers will be denoted, respectively, by
$ord_{x}\theta$ and $ord_{\partial}\theta$. For any $r, s\in {\bf N}$, the set of all $\theta\in \Theta$ such that $ord_{x}\theta\leq r $ and $ord_{\partial}\theta\leq s$ will be denoted by $\Theta(r, s)$.

If $D = \D\sum_{\alpha, \beta} k_{\alpha \beta}x^{\alpha}\partial^{\beta}\in A_{n}(K)$ (the sum is finite and $k_{\alpha \beta}\neq 0$ for any $\alpha, \beta\in {\bf N}^{n}$), then the $x$-order $ord_{x}D$ and $\partial$-order $ord_{\partial}D$ of $D$ are defined, as follows: $ord_{x}D = \max\{|\alpha|\,|\, k_{\alpha \beta}\neq 0\}$ and $ord_{\partial}D = \max\{|\beta|\,|\, k_{\alpha \beta}\neq 0\}$. These notions allow one to consider the Weyl algebra $A_{n}(K)$ as a bifiltered ring with the bifiltration $(W_{rs})_{r,s\in{\bf Z}}$ where $W_{rs} = \{D\in A_{n}(K)\,|\,ord_{x}D\leq r,\, ord_{\partial}D\leq s\}$ for any $(r, s)\in {\bf N}^{2}$ and $W_{rs} = 0$ for all $(r, s)\in{\bf Z}^{2}\setminus{\bf N}^{2}$. Clearly, $\bigcup{\{W_{rs}|r,s\in {\bf Z}\}}=A_{n}(K)$, $W_{rs} \subseteq W_{r+1,s}$ and $W_{rs} \subseteq W_{r,s+1}$ for any
$r,s\in {\bf Z}$. Furthermore,  $W_{rs}W_{kl}\subseteq W_{r+k,s+l}$ for any $r, s, k, l\in {\bf Z}$ and the last inclusion becomes an equality if $r, s, k, l\in {\bf N}$.

We shall consider two orderings $<_{x}$ and $<_{\partial}$ of the set $\Theta$ defined as follows: if $\theta = x^{\alpha}\partial^{\beta} = x_{1}^{\alpha_{1}}\dots x_{n}^{\alpha_{n}}\partial_{1}^{\beta_{1}}\dots \partial_{n}^{\beta_{n}}$ and $\theta' = x^{\gamma}\partial^{\delta} =
x_{1}^{\gamma_{1}}\dots x_{n}^{\gamma_{n}}\partial_{1}^{\delta_{1}}
\dots \partial_{n}^{\delta_{n}}$ are two elements of $\Theta$, then $\theta <_{x} \theta'$ if and only if $(ord_{x}\theta, ord_{\partial}\theta, \alpha_{1},\dots, \alpha_{n}, \beta_{1},\dots, \beta_{n})$ is less that $(ord_{x}\theta', ord_{\partial}\theta', \gamma_{1},\dots, \gamma_{n}, \delta_{1},\dots, \delta_{n})$ with respect to the lexicographic order on ${\bf N}^{2n+2}$, and similarly $\theta <_{\partial} \theta'$ if and only if $(ord_{\partial}\theta, ord_{x}\theta, \beta_{1},\dots, \beta_{n}, \alpha_{1},$

\noindent$\dots, \alpha_{n})$ is less that $(ord_{\partial}\theta', ord_{x}\theta',  \delta_{1},\dots, \delta_{n}, \gamma_{1},\dots, \gamma_{n})$ with respect to the lexicographic order on ${\bf N}^{2n+2}$.

\medskip

Let $\theta = x^{\alpha}\partial^{\beta} = x_{1}^{\alpha_{1}}\dots x_{n}^{\alpha_{n}}\partial_{1}^{\beta_{1}}\dots \partial_{n}^{\beta_{n}}$,
$\theta' = x^{\gamma}\partial^{\delta} = x_{1}^{\gamma_{1}}\dots x_{n}^{\gamma_{n}}\partial_{1}^{\delta_{1}}\dots \partial_{n}^{\delta_{n}}\in \Theta$. We say that $\theta$ divides $\theta'$ if $x^{\alpha}$ divides $x^{\gamma}$ and $\partial^{\beta}$ divides $\partial^{\delta}$, that is,
$\alpha_{i}\leq \gamma_{i}$ and $\beta_{i}\leq \delta_{i}$ for $i=1,\dots, n$. In this case we also say that $\theta'$ is a {\em multiple} of $\theta$ and write $\theta\,|\,\theta'$.

It is easy to see that if $\theta\,|\,\theta'$, then there exist elements
$\theta_{0}, \theta_{1},\dots, \theta_{k}\in \Theta$ such that
$\theta' = \theta_{0}\theta - \sum_{i=1}^{k}\theta_{i}$ where
$ord_{x}\theta_{0} + ord_{x}\theta = ord_{x}\theta'$,
$ord_{\partial}\theta_{0} + ord_{\partial}\theta = ord_{\partial}\theta'$, and
$ord_{x}\theta_{i}  < ord_{x}\theta'$,
$ord_{\partial}\theta_{i} < ord_{\partial}\theta'$ for $i=1,\dots, k$. In this case, we denote the
monomial $\theta_{0}$ by $\D\frac{\theta'}{\theta}$.

For example, if $n=1$, then $\theta = x\partial^{2}$ divides $\theta' = x^{2}\partial^{3}$ and
one can write $\theta' = \theta_{0}\theta - \theta_{1}$ where $\theta_{0} = x\partial$ and $\theta_{1} = x\partial^{2}$.

\medskip

In what follows, by the {\em least common multiple} of two elements $\theta', \theta''\in \Theta$ we mean the element $lcm(\theta', \theta'') =
lcm(\theta_{x}', \theta_{x}'')lcm(\theta_{\partial}', \theta_{\partial}'')$. it is easy to see that if $\theta = lcm(\theta', \theta'')$, then
$\theta'|\theta$, $\theta''|\theta$ and whenever $\theta'|\tau$, $\theta''|\tau$ for some $\tau\in \Theta$, one has $\theta|\tau$.

\medskip

Let $E$ be a finitely generated free $A_{n}(K)$-module with free generators $e_{1},\dots, e_{m}$. Obviously, $E$ can be considered as a vector $K$-space with the basis $\Theta e = \{\theta e_{i} | \theta \in \Theta, 1\leq i\leq p\}$ whose elements will be called {\em terms\/}. For any term
$\theta e_{j}$ ($\theta\in \Theta,\, 1\leq j\leq m$), we define the $x$-order $ord_{x}(\theta e_{j})$ and  $\partial$-order $ord_{\partial}(\theta e_{j})$ of this term  as the numbers $ord_{x}\theta$ and $ord_{\partial}\theta$, respectively. If $T\subseteq \Theta$, then the set $\{te_{i}\,|\,t\in T,\, 1\leq i\leq m\}$ will be denoted by $Te$ in particular, for any $r_{1}, r_{2}\in {\bf N}$, $\Theta(r_{1}, r_{2})e$ will denote the set $\{\theta e_{i}\,|\,ord_{x}\theta\leq_{x}r_{1},\, ord_{\partial}\theta\leq_{\partial}r_{2},\,  1\leq i\leq m\}$.

Since the set of all terms $\Theta e$ is a basis of the vector $K$-space $E$, every nonzero element $f\in E$ has a unique representation of the form
\begin{equation}
f = a_{1}\theta_{1}e_{i_{1}} + \dots + a_{s}\theta_{s}e_{i_{s}}
\end{equation}
where $\theta_{j}\in\in\Theta$, $a_{j}\in K$, $a_{j}\neq 0$ ($1\leq j\leq s$) and the terms $\theta_{1}e_{i_{1}},\dots, \theta_{s}e_{i_{s}}$ are all distinct. We say that a term $u$ {\em appears in $f$} (or that $f$ {\em contains} $u$) if $u$ is one of the terms $\theta_{k}e_{i_{k}}$ in the representation (5.1) (that is, the coefficient of $u$ in $f$ is not zero).

A term $u = \theta'e_{i}$ is said to be a {\em multiple} of a term $v = \theta e_{j}$ ($\theta, \theta'\in\Theta,\, 1\leq i, j\leq m$) if $i=j$ and $\theta|\theta'$. In this case we also say that $v$ divides $u$, write $v|u$ and set ${\D\frac{u}{v}} = {\D\frac{\theta'}{\theta}}$. The {\em least common multiple} of two terms $w_{1} = \theta_{1}e_{i}$ and $w_{2} = \theta_{2}e_{j}$ is defined as $$lcm(w_{1}, w_{2})=
\begin{cases}
lcm(\theta_1,\theta_2)e_i,&\textnormal{ if }i=j,\\
0,&\textnormal{ if }i\neq j.
\end{cases}$$
We shall consider two orderings of the set $\Theta e$ that correspond to the orderings $<_{x}$ and $<_{\partial}$ of $\Theta$. These orderings of $\Theta e$, which will be denoted by the same symbols $<_{x}$ and $<_{\partial}$, are defined as follows: if $\theta e_{i},\, \theta'e_{j}\in\Theta e$, then $\theta e_{i} <_{x} \theta'e_{j}$ (respectively, $\theta e_{i} <_{\partial} \theta'e_{j}$) if and only if $\theta <_{x} \theta'$ (respectively, $\theta <_{\partial} \theta'$) or $\theta = \theta'$ and $i < j$.
\begin{definition}
Let $f$ be a nonzero element of $E$ written in the form {\em (5.1).} Then the greatest with respect to $<_{x}$ term of the set $\{\theta_{1}e_{i_{1}}, \dots, \theta_{s}e_{i_{s}}\}$ is called the $x$-leader of $f$ while the greatest with respect to $<_{\partial}$ term of this set is called the $\partial$-leader of the element $f$. The $x$-leader and $\partial$-leader of $f$ will be denoted by $u_{f}$ and $v_{f}$, respectively. Furthermore, $lc_{x}(f)$ and $lc_{\partial}(f)$ will denote, respectively, the coefficients of $u_{f}$ and $v_{f}$ in representation {\em (4.1). (Of course, it is possible, that $u_{f}=v_{f}$ and therefore $lc_{x}(f)=lc_{\partial}(f)$.)}
\end{definition}
\begin{definition}
Let $f, g\in E$, $g\neq 0$. We say that $f$ is $(x,\partial)$-reduced with respect to $g$ if $f$ does not contain any multiple $\theta u_{g}$ of $u_{g}$ ($\theta\in\Theta$) such that $ord_{\partial}(\theta v_{g})\leq ord_{\partial}v_{f}$. An element $f\in E$ is said to be $(x,\partial)$-reduced with respect to a set $G\subseteq E$ if $f$ is $(x,\partial)$-reduced with respect to every element of $G$.
\end{definition}
Let us consider a new symbol $z$ and the free commutative semigroup $T$ of all power products $\theta z = x_{1}^{i_{1}}\dots x_{n}^{i_{n}}\partial_{1}^{j_{1}}\dots \partial_{n}^{j_{n}}z^{k}$ ($\theta = x_{1}^{i_{1}}\dots x_{n}^{i_{n}}\partial_{1}^{j_{1}}\in \Theta$, $k\in {\bf N}$). Let $Te = T\times \{e_{1},\dots, e_{m}\} = \{te_{i}\,|\,t\in T, \,1\leq i\leq m\}$. We say that an element $te_{i}\in Te$ divides an element $t'e_{j}\in Te$ and write $te_{i}|t'e_{j}$ if and only if $i=j$ and $t|t'$ in $T$ (if $t = \theta_{1}z^{k}$ and $t' = \theta_{2}z^{l}$, where $\theta_{1}, \theta_{2}\in\Theta$, then $t|t'$ means $\theta_{1}|\theta_{2}$ and $k\leq l$). Furthermore, for any $f\in E$, we set $d(f) = ord_{\partial}v_{f} - ord_{\partial}u_{f}$ and define the mapping $\rho"E\rightarrow Te$ by $\rho(f) = z^{d(f)}u_{f}$.
\begin{definition}
With the above notation, let $N$ be an $A_{n}(K)$-submodule of a free $A_{n}(K)$-module $E$ with a basis $\{e_{1},\dots, e_{m}\}$. A finite set of nonzero elements $G = \{g_{1},\dots, g_{r}\}$ is called a $(x,\partial)$-Gr\"obner basis of $N$ if for any nonzero element $f\in N$, there exists $g_{i}\in G$ such that $\rho(g_{i})|\rho(f)$.
\end{definition}
Since the condition $\rho(g_{i})|\rho(f)$ implies that $u_{g_{i}}|u_{f}$, any $(x,\partial)$-Gr\"obner basis of an $A_{n}(K)$-submodule $N$ of $E$ is a Gr\"obner basis of $N$ with respect to the total order $<_{x}$ in the usual sense.

A finite set of nonzero elements $G = \{g_{1},\dots, g_{r}\}\subseteq E$ is said to be a $(x,\partial)$-Gr\"obner basis if $G$ is a $(x,\partial)$-Gr\"obner basis of the $A_{n}(K)$-submodule $N = \D\sum_{i=1}^{r}A_{n}(K)g_{i}$ of $E$.
\begin{definition}
Given $f, g, h\in E$ with $g\neq 0$, we say that the element $f$ $(x,\partial)$-reduces to $h$ modulo $g$ in one step and write
$f\xrightarrow[x,\partial]{\text{g}} h$ if and only if $f$ contains some term $w$ with coefficient $a\neq 0$ such that $u_{g}|w$,
$$h = f - a(lc_{x}(g))^{-1}{\frac{w}{u_{g}}}g,\,\,\, \text{{\em and}}\,\,\,\, ord_{\partial}\left({\frac{w}{u_{g}}}v_{g}\right) \leq ord_{\partial}v_{f}.$$
\end{definition}
\begin{definition}
Let $f, h\in E$ and let $G = \{g_{1},\dots, g_{r}\}$ be a finite set of nonzero elements of $E$. We say that $f$ is $(x,\partial)$-reduces to $h$ modulo $G$ and write $f\xrightarrow[x,\partial]{\text{G}} h$ if and only if there exist elements $g^{(1)}, g^{(2)},\dots g^{(p)}\in G$ and
$h^{(1)},\dots, h^{(p-1)}\in E$ such that
$$f\xrightarrow[x,\partial]{\text{$g^{(1)}$}} h^{(1)}\xrightarrow[x,\partial] {\text{$g^{(2)}$}} \dots \xrightarrow[x,\partial] {\text{$g^{(p-1)}$}}h^{(p-1)}\xrightarrow[x,\partial] {\text{$g^{(p)}$}}h.$$
\end{definition}
\begin{theorem}
With the above notation, let $f\in E$ and let $G = \{g_{1},\dots, g_{r}\}$ be an $(x,\partial)$-Gr\"obner basis in $E$. Then there exist elements $g\in E$ and $Q_{1},\dots, Q_{r}\in A_{n}(K)$ such that $f - g = \D\sum_{i=1}^{r}Q_{i}g_{i}$ and $g$ is $(x,\partial)$-reduced with respect to $G$.
\end{theorem}
\begin{proof}
If $f$ is $(x,\partial)$-reduced with respect to $G$, the statement is obvious (one can set $g=f$). Suppose that $f$ is not $(x,\partial)$-reduced with respect to $G$. Let $u_{i} = u_{g_{i}}$, $v_{i} = v_{g_{i}}$, and $a_{i} = lc_{x}(g_{i})$ ($1\leq i\leq r$). In what follows, a term $w_{h}$ will be called a {\em $G$-leader} of an element $h\in E$ if $w_{h}$ is the greatest with respect to $<_{x}$ term among all terms $w$ with the following properties:

(i)\, $w$ appears in $h$;

(ii)\, $w$ is a multiple some $u_{i}$ ($1\leq i\leq r$) and $ord_{\partial}\left({\D\frac{w}{u_{i}}}v_{i}\right)\leq ord_{\partial}v_{h}$.

\medskip

Let $w_{f}$ be the $G$-leader of an element $f\in E$ and let $c_{f}$ be the coefficient of $w_{f}$ in representation (5.1) of $f$. Then $u_{i}|w$ for some $i$, $1\leq i\leq r$, and $ord_{\partial}\left({\D\frac{w}{u_{i}}}v_{i}\right)\leq ord_{\partial}v_{f}$. Without loss of generality we can assume that $i$ corresponds to the greatest with respect to $<_{x}$ $x$-leader $u_{i}$ satisfying the above condition. Let $g'_{i} = g_{i} - lx_{x}(g_{i})u_{i}$ and let ${\D\frac{w_{f}}{u_{i}}}=\theta$, so that $w_{f} = \theta u_{i} + S_{i}$ where $S_{i}$ denotes a sum of terms of the form $\theta'e_{k}$ ($\theta'\in\Theta,\, 1\leq k\leq m$) such that $ord_{x}(\theta'e_{k}) < ord_{x}(\theta u_{i}) = ord_{x}w_{f}$ and $ord_{\partial}(\theta'e_{k}) < ord_{\partial}(\theta u_{i}) = ord_{\partial}w_{f}$.

Let $f' = f - c_{f}\left(lc_{x}(g_{i})\right)^{-1}\theta g_{i} = f - c_{f}\left(lc_{x}(g_{i})\right)^{-1}(lc_{x}(g_{i})(w_{f} - S_{i}) + \theta g'_{i}) = f - c_{f}w_{f} + c_{f}S_{i} - c_{f}\left(lc_{x}(g_{i})\right)^{-1}\theta g'_{i}$ where the $G$-leader of $c_{f}\left(lc_{x}(g_{i})\right)^{-1}S_{i} - c_{f}\left(lc_{x}(g_{i})\right)^{-1}\theta g'_{i}$ is less than $\theta u_{i}$, and therefore less than $w_{f}$, with respect to $<_{x}$. Clearly, $f'$ does not contain $w_{f}$ and $ord_{\partial}v_{f'}\leq ord_{\partial}v_{f}$ (since $ord_{\partial}\theta g'_{i}\leq ord_{\partial}\theta g'_{i} = ord_{\partial}\theta g_{i} = ord_{\partial}\left({\D\frac{w_{f}}{u_{i}}}v_{i}\right) \leq ord_{\partial}v_{f}$ and $ord_{\partial}S_{i}\leq ord_{\partial}w_{f}\leq ord_{\partial}v_{f}$). Furthermore, $f'$ cannot contain any term $w'$ such that $w_{f} <_{x} w'$, $u_{i}|w'$, and $ord_{\partial}\left({\D\frac{w_{f}}{u_{i}}}v_{i}\right) \leq ord_{\partial}v_{f'}$. Indeed, if such a term $w'$ appears in $f'$, then we would have $ord_{\partial}\left({\D\frac{w_{f}}{u_{i}}}v_{i}\right) \leq ord_{\partial}v_{f}$ (clearly, $ord_{\partial}v_{f'}\leq ord_{\partial}v_{f}$ because, as we have seen,  $ord_{\partial}\theta g_{i} = ord_{\partial}\left({\D\frac{w_{f}}{u_{i}}}v_{i}\right)\leq ord_{\partial}v_{f}$), so $w'$ cannot appear in $f$ by the choice of $w_{f}$. The term $w'$ cannot appear in $\theta g_{i}$ either, since $u_{\theta g_{i}} = w_{f} <_{x}w'$. Thus, $w'$ cannot appear  in $f' = f - c_{f}\left(lc_{x}(g_{i})\right)^{-1}\theta g_{i}$, hence the $G$-leader $w_{f'}$ of $f'$ is strictly less than $w_{f}$ with respect to $<_{x}$. Applying the same procedure to $f'$ and continuing in the same way we will obtain an element $g\in E$ such that $f - g \in \D\sum_{i=1}^{r}A_{n}(K)g_{i}$ and $g$ is $(x,\partial)$-reduced with respect to $G$.
\end{proof}

The process of reduction described in the proof of Theorem 5.6 can be realized with the following algorithm.

\begin{algorithm} ($f, r, g_{1},\dots, g_{r};\, g$)

{\bf Input:} $f\in E$, a positive integer $r$, $G = \{g_{1},\dots,
g_{r}\}\subseteq E$  where $g_{i}\ne 0$ for $i = 1,\dots, r$

{\bf Output:} Elements $g\in E$ and $Q_{1},\dots, Q_{r}\in D$ such
that $g = f - \sum_{i=1}^{r}Q_{i}g_{i}$\, and $g$ is reduced with
respect to $G$

{\bf Begin}\,\,\, $Q_{1}:= 0, \dots, Q_{r}:= 0, g:= f$

{\bf While} there exist $i$, $1\leq i\leq r$, and a term $w$, that
appears in $g_{i}$  with a nonzero coefficient $c(w)$, such that
$u_{g_{i}}|w$ and
$ord_{\partial}(\frac{w}{u_{g_{i}}}v_{g_{i}})\leq
ord_{\partial}v_{g}$ {\bf do}

$z$:=  the greatest (with respect to $<_{x}$)  term $w$
satisfying the above conditions.

$k$:= the smallest number $i$ for which $u_{g_{i}}$ is the
greatest (with respect to $<_{x}$) $x$-leader of an element $g_{i}
\in G$ such that $u_{g_{i}}|z$ and
$ord_{\partial}(\frac{z}{u_{g_{i}}}v_{g_{i}}) \leq
ord_{\partial}v_{g}$.

$Q_{k}:= Q_{k} + c(z)\left(lc_{x}(g_{k})\right)^{-1}{\D\frac{z}{u_{g_{k}}}}g_{k}$;
\,\, $g:= g - c(z)\left(lc_{x}(g_{k})\right)^{-1}{\D\frac{z}{u_{g_{k}}}}g_{k}$.
\end{algorithm}
The proof of Theorem 4.6 shows that if $G$ is an $(x,\partial)$-Gr\"obner basis of an $A_{n}(K)$-submodule $N$ of $E$, then the reduction step described in Definition 4.4 can be applied to every nonzero element of $f\in N$. As a result of such a step, we obtain an element of $N$ whose $G$-leader is strictly less than the $G$-leader of $f$ with respect to $<_{x}$. This observation leads to the following statement.
\begin{theorem}
Let $G = \{g_{1},\dots, g_{r}\}$ be an $(x,\partial)$-Gr\"obner basis of an $A_{n}(K)$-submodule $N$ of $E$. Then

{\em (i)}\, $f\in N$ if and only if $f\xrightarrow[x,\partial]{\text{$G$}} 0$\,.

{\em (ii)}\, If $f\in N$ and $f$ is $(x,\partial)$-reduced with respect to $G$, then $f=0$.
\end{theorem}
\begin{definition} Let $f$ and $g$ be two elements in the free $A_{n}(K)$-module $E$. Let $\theta^{(x)}_{f} = {\D\frac{lcm(u_{f}, u_{g})}{u_{f}}}$,\,\,$\theta^{(\partial)}_{f} = {\D\frac{lcm(v_{f}, v_{g})}{v_{f}}}$, $\theta^{(x)}_{g} = {\D\frac{lcm(u_{f}, u_{g})}{u_{g}}}$,\,and \,$\theta^{(\partial)}_{g} = {\D\frac{lcm(v_{f}, v_{g})}{v_{g}}}$.
Then the elements $S_{x}(f, g) = \left(lc_{x}(f)\right)^{-1}\theta^{(x)}_{f}f -
\left(lc_{x}(g)\right)^{-1}\theta^{(x)}_{g}g$ and $S_{\partial}(f, g) = \left(lc_{\partial}(f)\right)^{-1}\theta^{(\partial)}_{f}f -
\left(lc_{\partial}(g)\right)^{-1}\theta^{(\partial)}_{g}g$ are called the $x$-$S$-polynomial and  $\partial$-$S$-polynomial of $f$ and $g$, respectively.
\end{definition}

\begin{theorem}
With the above notation, let  $f, g_{1},\dots, g_{r}\in E$ ($r\geq 1$) and let $f = \D\sum_{i=1}^{r}c_{i}\theta_{i}g_{i}$ where $\theta_{i}\in
\Theta$, $c_{i}\in K$ ($1\leq i\leq r$). Let $u_{\nu j} = lcm(u_{g_{\nu}}, u_{g_{j}})$ for any $\nu, j\in \{1,\dots, r\}$.  Furthermore, suppose
that $\theta_{1}u_{g_{1}} = \dots = \theta_{r}u_{g_{r}} = u$, $u_{f} <_{x} u$ and $\theta_{i}v_{g_{i}}\leq_{\partial} v_{f}$ for all $i\in
\{1,\dots, r\}$. Then there exist elements $c_{\nu j}\in K$ ($1\leq \nu\leq s, 1\leq j\leq t$) such that
$f = \D\sum_{\nu=1}^{s}\D\sum_{j =1}^{t}c_{\nu j}\theta_{\nu j}S_{x}(g_{\nu}, g_{j})$ where $\theta_{\nu j} = \D\frac{u}{u_{\nu j}}$ and
$\theta_{\nu j}u_{S_{x}(g_{\nu}, g_{j})} <_{x}u$,\,\,$\theta_{\nu j}v_{S_{x}(g_{\nu}, g_{j})}\leq_{\partial} v_{f}$\, ($1\leq \nu \leq s,\, 1\leq
j\leq t$).
\end{theorem}
\begin{proof}
 Without loss of generality we can assume that $lc_{x}(g_i)=1$ for $i=1,\ldots,r$. Then the inequality $u_{f}<_x u$ implies that $c_1+\cdots+c_r=0$. Furthermore,  $$S_x(g_{\nu},g_{j})=\frac{u_{\nu j}}{ u_{g_\nu}}g_{\nu}-\frac{u_{\nu j}}{u_{g_j}}g_{j},$$ for any $\nu, j\in \{1,\dots, r\}$, and for every $i=2,\ldots,r-1$ we have$$\frac{u}{ u_{g_i}}g_i=\frac{u}{u_{i-1,i}}\frac{u_{i-1,i}}{u_{g_i}}g_i=\frac{u}{u_{i,i+1}}\frac{u_{i,i+1}}{u_{g_i}}g_i.$$Using these equalities and the equalities
\begin{eqnarray*}
 \frac{u}{ u_{g_1}}g_1&=&\frac{u}{u_{1,2}}\frac{u_{1,2}}{ u_{g_1}}g_1,\nonumber\\
\frac{u}{u_{g_r}}g_r&=&\frac{u}{u_{r-1,r}}\frac{u_{r-1,r}}{ u_{g_r}}g_r\nonumber\\
\end{eqnarray*}
we obtain
\begin{eqnarray*}
 f&=&c_1\lambda_1g_1+\cdots c_r\lambda_rf_r\nonumber\\
&=&c_1\frac{u}{ u_{g_1}}g_1+\cdots+c_{r}\frac{u}{ u_{g_r}}g_{r}\nonumber\\
&=&c_1\theta_{1,2}\left(\frac{ u_{1,2}}{ u_{g_1}}g_1-\frac{ u_{1,2}}{ u_{g_2}}g_2\right)\nonumber\\
&&+(c_1+c_2)\theta_{2,3}\left(\frac{ u_{2,3}}{ u_{g_2}}g_2-\frac{ u_{2,3}}{ u_{g_3}}g_3\right)+\cdots\nonumber\\
&&+(c_1+\cdots+c_{r-1})\theta_{r-1,r}\left(\frac{ u_{r-1,r}}{ u_{g_{r-1}}}g_{r-1}-\frac{ u_{r-1,r}}{u_{g_r}}g_r\right)\nonumber\\
&&+(c_1+\cdots+c_{r})\frac{u}{u_{g_r}}\nonumber\\
&=&c_1\theta_{1,2}S_x(g_1,g_2)+(c_1+c_2)\theta_{2,3}S_x(g_2,g_3)+\cdots\nonumber\\
&&+(c_1+\cdots+c_{r-1})\theta_{r-1,r}S_x(g_{r-1},g_r).\nonumber
\end{eqnarray*}
Since  $u_{\theta_{i-1,i}S_x(g_{i-1},g_{i})}<_x u$, and  $v_{\theta_{i-1,i}S_x(g_{i-1},g_{i})}\leq_\partial v_{f}$ for all $i=2,\ldots,r$, we have
the desired representation of $f$.
\end{proof}

The following result provides the theoretical foundation for the algorithm for constructing $(x,\partial)$-Gr\"obner bases.
\begin{theorem} With the above notation, let $G = \{g_{1},\dots, g_{r}\}$ be a Gr\"obner basis of an $A_{n}(K)$-submodule $N$ of $E$ with
respect to the order $<_{\partial}$. Furthermore, suppose that $S_{x}(g_{i}, g_{j}) \xrightarrow[x,\partial]{\text{$G$}} 0$ for any $g_{i}, g_{j}\in G$. Then $G$ is an $(x,\partial)$-Gr\"obner basis of $N$.
\end{theorem}
\begin{proof}
Notice that it is sufficient to prove that under the conditions of the theorem every element $f\in N$ can be represented as
\begin{equation}
f = \sum_{i=1}^{r}h_{i}g_{i}
\end{equation}
where $h_{1},\dots, h_{r}\in A_{n}(K)$,
\begin{equation}
max_{<_{x}}\{u_{h_{i}}u_{g_{i}}\,|\,1\leq i\leq r\} = u_{f},
\end{equation}
(symbol $\max_{<_{x}}$ indicates that the maximum is taken with respect to the order $<_{x}$) and
\begin{equation}
ord_{\partial}(v_{h_{i}}v_{g_{i}})\leq ord_{\partial}v_{f}.
\end{equation}
Indeed, with the notation of Definition 4.3, if the above conditions hold, then $\rho(f)$ is divisible by $\rho(g_{i})$ where $g_{i}$ gives the maximum value in the left-hand side of (4.3).

Let $f\in N$. Since $G$ is a Gr\"obner basis with respect to $<_{\partial}$, one can write $f$ as
\begin{equation}
f = \sum_{i=1}^{r}H_{i}g_{i}
\end{equation}
where $H_{1},\dots, H_{r}\in A_{n}(K)$ and
\begin{equation}
max_{<_{\partial}}\{v_{H_{i}}v_{g_{i}}\,|\,1\leq i\leq r\} = v_{f},
\end{equation}

Let us choose among all representations of the form (4.5) with condition (4.6) a representation with the smallest with respect to $<_{x}$ possible term $u = \max_{<_{x}}\{u_{H_{i}}u_{g_{i}}\,|\,1\leq i\leq r\}$. Setting $d_{i} = lc_{k}(H_{i})$ ($1\leq i\leq r$) and breaking the sum (4.5) in three parts we can write
\begin{equation}
f = \sum_{u_{H_{i}}u_{g_{i}}=u}d_{i}u_{H_{i}}g_{i} + \sum_{u_{H_{i}}u_{g_{i}}=u}(H_{i} - d_{i}u_{H_{i}})g_{i} + \sum_{u_{H_{i}}u_{g_{i}}<_{x} u}H_{i}g_{i}\,.
\end{equation}
Note that if $u = u_{f}$, then the expression (4.7) satisfies conditions (4.2) - (4.4). Indeed, by (4.6) we have $v_{f} = \max_{<_{\partial}}\{v_{H_{i}}v_{g_{i}}\,|\,1\leq i\leq r\}$, hence $\max\{ord_{\partial}(v_{H_{i}}v_{g_{i}})\,|\,1\leq i\leq r\}\leq ord_{\partial}v_{f}$. Suppose that $u_{f} <_{x} u$. Since $u = \max_{<_{x}}\{u_{H_{i}}u_{g_{i}}\,|\,1\leq i\leq r\}$, we have $u_{H_{i}-d_{i}H_{i}} <_{x} u$ ($1\leq i\leq r$) whence the $x$-leader of the second sum in (4.7) does not exceed $u$ with respect to $<_{x}$. Furthermore, it is clear that $u_{H_{i}}u_{g_{i}}=u$ for any term in the sum
\begin{equation}
\tilde{f} = \sum_{u_{H_{i}}u_{g_{i}}=u}d_{i}u_{H_{i}}g_{i}
\end{equation}
and $ord_{\partial}v_{\tilde{f}}\leq \max_{i\in I}\{ord_{\partial}(v_{H_{i}}v_{g_{i}})\}\leq ord_{\partial}v_{f}$ where $I$ denotes the set of all indices $i\in \{1,\dots, r\}$ that appear in (4.8).

Let $u_{ij} = lcm(u_{i}, u_{j})$ for any $i, j\in I$, $i\neq j$, and let $\theta_{ij}={\D\frac{u}{u_{ij}}}\in\Theta$ ($u_{ij}|u$, since $u= u_{H_{i}}u_{g_{i}}$ for every $i\in I$.) By Theorem 4.10, there exist elements $c_{ij}\in K$ such that
\begin{equation}
\tilde{f} = \sum_{i, j}c_{ij}\theta_{ij}S_{x}(g_{i}, g_{j})
\end{equation}
where $u_{\theta_{ij}S_{x}(g_{i}, g_{j})} <_{x} u_{\tilde{f}} = u$ and $ord_{\partial}v_{\theta_{ij}S_{x}(g_{i}, g_{j})}\leq ord_{\partial}v_{\tilde{f}}$.

Since $S_{x}(g_{i}, g_{j})\xrightarrow[x,\partial]{\text{$G$}} 0$, there exist $q_{\nu ij}\in A_{n}(K)$ such that $S_{x}(g_{i}, g_{j}) = \D\sum_{\nu = 1}^{r}q_{\nu ij}g_{\nu}$ and $u_{q_{\nu ij}}u_{g_{\nu}}\leq_{x}u_{S_{x}(g_{i}, g_{j})}$,\, $ord_{\partial}(v_{q_{\nu ij}}v_{g_{\nu}}\leq ord_{\partial}v_{S_{x}(g_{i}, g_{j})}$.

Thus, for any indices $i, j$ in sum (4.9), one has $$\theta_{ij}S_{x}(g_{i}, g_{j}) = \sum_{\nu = 1}^{r}(\theta_{ij}q_{\nu ij})g_{\nu}$$
where $u_{\theta_{ij}q_{\nu ij}}u_{g_{\nu}} = \theta_{ij}u_{q_{\nu ij}}u_{g_{\nu}}\leq_{x}\theta_{ij}u_{S_{x}(g_{i}, g_{j})} <_{x} u$. It follows that \begin{equation}
\tilde{f} = \sum_{i, j}c_{ij}\sum_{\nu = 1}^{r}(\theta_{ij}q_{\nu ij})g_{\nu} = \sum_{\nu = 1}^{r}\left(\sum_{i, j}c_{ij}\theta_{ij}q_{\nu ij}\right)g_{\nu} = \sum_{\nu = 1}^{r}\tilde{H}_{\nu}g_{\nu}
\end{equation}
where $\tilde{H}_{\nu} = \sum_{i, j}c_{ij}\theta_{ij}q_{\nu ij}$ and $u_{\tilde{H}_{\nu}}u_{g_{\nu}} <_{x} u$\, ($1\leq \nu\leq r$). Furthermore,

\medskip

$ord_{\partial}(v_{\tilde{H}_{\nu}}v_{g_{\nu}})\leq \max_{i, j}\{ord_{\partial}(\theta_{ij}q_{\nu ij}v_{g_{\nu}})\}\leq$

$\max_{i, j}\{\max\{ord_{\partial}\left(\theta_{ij}{\D\frac{u_{ij}}{u_{g_{i}}}}v_{g_{i}}\right), \left(\theta_{ij}{\D\frac{u_{ij}}{u_{g_{j}}}}v_{g_{j}}\right)\}\} = $

$\max_{i, j}\{\max\{ord_{\partial}\left({\D\frac{u}{u_{g_{i}}}}v_{g_{i}}\right), ord_{\partial}\left({\D\frac{u}{u_{g_{j}}}}v_{g_{j}}\right)\leq ord_{\partial}u =$

$ord_{\partial}u_{\tilde{f}}\leq ord_{\partial}v_{\tilde{f}}$,\, so that representation (4.10) satisfies the condition
\begin{equation}
ord_{\partial}(v_{\tilde{H}_{\nu}}v_{g_{\nu}})\leq ord_{\partial}v_{\tilde{f}}\leq ord_{\partial}v_{f}
\end{equation}
for $\nu=1,\dots, r$. Substituting (4.10) into (4.7) we obtain
\begin{equation}
f = \sum_{\nu=1}^{r}\tilde{H}_{\nu}g_{\nu} + \sum_{u_{H_{i}}u_{g_{i}}=u}(H_{i} - d_{i}u_{H_{i}})g_{i} + \sum_{u_{H_{j}}u_{g_{j}}<_{x} u}H_{j}g_{j}
\end{equation}
where, denoting each $H_{i} - d_{i}u_{H_{i}}$ in the second sum by $H'_{i}$, we have the following conditions: $u_{\tilde{H}_{\nu}}u_{g_{\nu}} <_{x}u$ ($1\leq \nu\leq r$), $u_{H'_{i}}u_{g_{i}}<_{x}u$ for any term with index $i$ in the second sum of (4.12), and $u_{H_{j}}u_{g_{j}}<_{x}u$ for any term with index $j$ in the third sum of (4.12). We also have inequality (4.11) for the first sum, the inequality $ord_{\partial}(v_{H'_{i}}v_{g_{i}})\leq ord_{\partial}(v_{H_{i}}v_{g_{i}})\leq ord_{\partial}v_{f}$ (see (4.6)\,) for every index $i$ in the second sum and the inequality $ord_{\partial}(v_{H_{j}}v_{g_{j}})\leq ord_{\partial}v_{f}$ for every index $j$ in the third sum of (4.12). Thus, (4.12) is a representation of $f$ in the form (4.5) with condition (4.6) such that if one writes (4.12) as $f = \D\sum_{i=1}^{r}\tilde{H}'_{i}g_{i}$ (combining the sums in (4.12)\,), then $\max_{<_x}\{u_{\tilde{H}'_{1}}u_{g_{1}}, \dots, u_{\tilde{H}'_{r}}u_{g_{r}}\} <_{x} u$ and one has condition (4.6). We have arrived at a contradiction with our choice of representation (4.5) with condition (4.6) and the smallest with respect to $<_{x}$ possible value of \ $u = \max\{u_{H_{i}}u_{g_{i}}\,|\,1\leq i\leq r\}$. Thus, every element $f\in N$ can be written in the form (4.2) with conditions (4.3) and (4.4). This completes the proof of the theorem.
\end{proof}
The last theorem allows one to construct an $(x, \partial)$-Gr\"obner basis of an $A_{n}(K)$-submodule of $E$ starting with the usual Gr\"obner basis of $N$ with respect to the term order $<_{\partial}$.

\section{Bivariate dimension polynomials of $A_{n}(K)$-modules and their invariants}

In what follows we consider the ring $A_{n}(K)$ as a bifiltered ring with respect to the natural bifiltration $(W_{rs})_{r,s\in {\bf Z}}$ introduced at the beginning of the preceding section. Recall that  $W_{rs}=0$, if at least one of the numbers $r$, $s$ is negative, and if $r\geq 0$, $s\geq 0$, then $W_{rs}$ is a vector $K$-space generated by the set $\Theta(r,s) = \{\theta \in \Theta | ord_{x}\theta \leq r, ord_{\partial}\theta \leq s\}$. It follows from the third statement of Proposition 3.3 that $dim_{K}W_{rs} = Card\,\Theta(r,s) = \D{{r+n}\choose m}{{s+n}\choose n}$ for any $r, s\in {\bf N}$.
\begin{definition}
Let $M$ be a module over a Weyl algebra $A_{n}(K)$.  A bisequence $(M_{rs})_{r,s\in {\bf Z}}$ of vector $K$-subspaces of the module $M$ is called a {\em bifiltration\/} of $M$ if the following three conditions hold:

{\em (i)}\, If $r\in {\bf Z}$ is fixed, then $M_{rs}\subseteq M_{r, s+1}$ for all $s\in {\bf Z}$ and  $M_{rs}=0$ for all sufficiently small $s\in {\bf Z}$. Similarly, if $s\in {\bf Z}$ is fixed, then $M_{rs}\subseteq M_{r+1, s}$ for all $r\in {\bf Z}$ and $M_{rs}=0$ for all sufficiently small $r\in {\bf Z}$.

{\em (ii)}\, $\bigcup \{M_{rs}|r,s\in {\bf Z}\}=M$.

{\em (iii)}\,  $W_{kl}M_{rs}\subseteq M_{r+k,s+l}$ for any $r,s\in {\bf Z}$,  $k,l\in {\bf N}$.
\end{definition}
\begin{example}
Let $M$ be a finitely generated $A_{n}(K)$-module with generators $f_{1},\dots,f_{m}$. Then the vector $K$-spaces $M_{rs}=\sum_{i=1}^{m} W_{rs}f_{i}$ ($r,s\in {\bf Z}$) form a bifiltration of the module $M$.  This bifiltration is called a {\em natural bifiltration\/} of $M$ associated with the system of generators $f_{1},\dots,f_{m}$. It is easy to see that every component $M_{rs}$ of the this bifiltration is a finitely generated vector $K$-space and $W_{kl}M_{rs} = M_{r+k, s+l}$ for any $r, s, k, l\in {\bf N}$.
\end{example}
In what follows we use the properties of $(x, \partial)$-Gr\"obner bases to prove the existence and obtain a method of computation of bivariate dimension polynomials of finitely generated $A_{n}(K)$-modules. The following result can be considered as the main step in this direction.
\begin{theorem}
Let $M$ be a finitely generated $A_{n}(K)$-module with a system of generators \{$f_{1}, \dots , f_{m}$\}, $E$ a free $A_{n}(K)$-module with a basis $e_{1}, \dots , e_{m}$, and $\pi:E \longrightarrow M$ the natural $A_{n}(K)$-epimorphism of $E$ onto $M$ $(\pi(e_{i})=f_{i}$ for $i=1,\dots,m)$.
Furthermore, let $N=Ker\,\pi$ and let $G = \{g_{1}, \dots , g_{d}\}$ be an $(x, \partial)$-Gr\"obner basis of $N$. Finally, for any $r, s \in {\bf N}$, let $M_{rs} =\sum_{i=1}^{m}W_{rs}f_{i}$, and let $U_{rs}$ denote the set \{$w\in \Theta e|ord_{x}w\leq r, ord_{\partial}w\leq s$, and either
$w$ is not a multiple of any $u_{g_{i}}$ $(1\leq i\leq d)$ or $ord_{\partial}(\theta v_{g_{j}})>s$ for any $\theta \in \Theta, g_{j}\in G$ such that $w = \theta u_{g_{j}}$\}. Then $\pi(U_{rs})$ is a basis of the vector $K$-space $M_{rs}$.
\end{theorem}

\begin{proof}
Let us prove, first, that every element $\theta f_{i}$ ($1\leq i\leq m$, $\theta \in \Theta(r,s)$), that does not belong to $\pi(U_{rs})$, can be written as a finite linear combination of elements of $\pi(U_{rs})$ with coefficients in $K$ (so that the set $\pi(U_{rs})$ generates the vector $K$-space $M_{rs}$). Since $\theta f_{i}\notin \pi(U_{rs})$, $\theta e_{i}\notin U_{rs}$ whence $\theta e_{i}=\theta'u_{g_{j}}$ for some $\theta'\in \Theta$, $1\leq j\leq d$, such that $ord_{\partial}(\theta'v_{g_{j}})\leq s$.  Let us consider the element $g_{j}=a_{j}u_{g_{j}}+ \dots $\, ($a_{j}\in K, a_{j}\ne 0$), where dots are placed instead of the other terms that appear in $g_{j}$ (obviously, those terms are less than $u_{g_{j}}$ with respect to the order $<_x$). Since $g_{j}\in N=Ker\,\pi$, $\pi(g_{j})=a_{j}\pi(u_{g_{j}})+ \dots = 0$, whence $\pi(\theta'g_{j})=a_{j}\pi(\theta'u_{g_{j}})+ \dots = a_{j}\pi(\theta e_{i})+ \dots = a_{j}\theta f_{i}+ \dots = 0$, so that $\theta f_{i}$ is a finite linear combination with coefficients in $K$ of some elements $\tilde{\theta}f_{k}$ ($1\leq k\leq m$) such that $\tilde{\theta}\in \Theta(r,s)$ and $\tilde{\theta}e_{k} <_{x} \theta'u_{g_{j}}$.
($ord_{x}\tilde{\theta}\leq r$, since $\tilde{\theta}e_{k} <_{x} \theta e_{i}$ and $\theta \in \Theta(r,s)$; $ord_{\partial}\tilde{\theta}\leq s$, since $\tilde{\theta}e_{k} \leq _{\partial} v_{\theta'g_{j}}=\theta'v_{g_{j}}$ and $ord_{\partial}(\theta'v_{g_{j}})\leq s$).  Thus, we can apply the
induction on $\theta e_{j}$ ($\theta \in \Theta, 1\leq j\leq p$) with
respect to the order $<_{x}$ and obtain that every element $\theta f_{i}$ ($\theta \in \Theta(r,s), 1\leq j\leq m$) can be written as a finite linear combination of elements of $\pi(U_{rs})$ with coefficients from the field $K$.

Now, let us prove that the set  $\pi(U_{rs})$ is linearly independent over $K$.  Let $\sum_{i=1}^{q}{a_{i}\pi(u_{i})} = 0$ for some $u_{1},\dots, u_{q}\in U_{rs},\,  a_{1},\dots, a_{q}\in K$. Then $h=\sum_{i=1}^{q}{a_{i}u_{i}}$ is an element of $N$ $(x, \partial)$-reduced with respect to $G$. Indeed, if an element $u=\theta e_{j}$ appears in $h$ (so that $u=u_{i}$ for some $i=1,\dots, q$), then either $u$ is not a multiple of any $u_{g_{j}}$ ($1\leq j\leq d$) or $u=\theta u_{g_{k}}$ for some $\theta \in \Theta,\, 1\leq k\leq d$, such that $ord_{\partial}(\theta v_{g_{k}}) > s \geq ord_{\partial}v_{h}$
(since $v_{h}$ is one of the elements $u_{1}, \dots, u_{q}$ that lie in $U_{rs}$). Applying Theorem 4.8 we obtain that $h = 0$, whence $a_{1}= \dots =a_{q} = 0$.  This completes the proof of the theorem.
\end{proof}
Theorem 5.3 leads to the following existence theorem, which is the main result of this section.
\begin{theorem}
Let $M$ be a finitely generated $A_{n}(K)$-module with a system of generators $\{f_{1}, \dots , f_{m}\}$ and let $(M_{rs})_{r, s\in {\bf Z}}$
be the corresponding natural bifiltration of $M$ $(M_{rs}=\sum_{i=1}^{p}W_{rs}f_{i}$ for $r, s\in {\bf Z}.)$ Then there exists a numerical polynomial
$\phi_{M}(t_{1},t_{2})$ in two variables $t_{1}, t_{2}$ such that

{\em (i)}\, $\phi_{M}(r,s)=dim_{K} M_{rs}$ for all sufficiently large $(r,s)\in {\bf Z}^2$. {\em (}It means that there exist $r_{0}, s_{0}\in {\bf Z}$ such that the equality holds for all $r\geq r_{0},\, s\geq s_{0}$.{\em)}

{\em (ii)}\, $deg_{t_{1}}\phi_{M}(t_{1},t_{2}) \leq n$ and $deg_{t_{2}}\phi_{M}(t_{1},t_{2}) \leq n$, so that $deg\,\phi_{M}(t_{1},t_{2})\leq 2n$ and the polynomial $\phi_{M}(t_{1},t_{2})$ can be represented as
\begin{equation}
\phi(t_{1}, t_{2}) = \sum_{i=0}^{n}\sum_{j=0}^{n}a_{ij}{{t_{1}+i}\choose i}{{t_{2}+j}\choose j}
\end{equation} where $a_{ij}\in {\bf Z}$ for all $i, j$.
\end{theorem}
\begin{proof}
Let $E$ be a free $A_{n}(K)$-module with a basis $e_{1}, \dots , e_{m}$, let $N$ be the kernel of the natural epimorphism $\pi:E \longrightarrow M$, and let the set $U_{rs}$ ($r, s\in {\bf N}$) be the same as in the conditions of Theorem 5.3. Furthermore, let $G = \{g_{1}, \dots , g_{d}\}$ be an  $(x, \partial)$-Gr\"obner basis of $N$. By Theorem 5.3, for any $r, s\in{\bf N}$,\,  $\pi(U_{rs})$ is a basis of the vector $K$-space  $M_{rs}$. Therefore, $dim_{K} M_{rs} = Card\,\pi(U_{rs}) = Card\,U_{rs}$. (It was shown in the second part of the proof of Theorem 5.3 that the restriction of the mapping $\pi$ on $U_{rs}$ is bijective.) 

Let $U'_{rs}$ = \{$w\in U_{rs}|$ $w$ is not a multiple of any element $u_{g_{i}}$ ($1\leq i\leq d$)\} and $U''_{rs}$ = \{$w\in U_{rs}|$ $w=\theta u_{g_{j}}$ for some $g_{j}$ ($1\leq j\leq d$) and $\theta \in \Theta$ such that $ord_{\partial}(\theta v_{g_{j}}) > s$\}. Then $U_{rs}=U'_{rs}\bigcup U''_{rs}$ and $U'_{rs}\bigcap U''_{rs} = \emptyset$, whence $Card\,U_{rs} = Card\,U'_{rs} +  Card\,U''_{rs}$.

By Proposition 3.3, there exists a numerical polynomial $\omega (t_{1}, t_{2})$ in two variables $t_{1}$ and $t_{2}$ such that $\omega (r, s) =
Card\,U'_{rs}$ for all sufficiently large  $(r,s)\in {\bf N}^2$. In order to express $Card\,U''_{rs}$ in terms of $r$ and $s$, let us set
$a_{i} = ord_{x}u_{g_{i}}$, $b_{i} = ord_{\partial}u_{g_{i}}$, $c_{i} = ord_{\partial}v_{g_{i}}$, $a_{ij} = ord_{x}lcm(u_{g_{i}}, u_{g_{j}})$,\,
$b_{ij} = ord_{\partial}lcm(u_{g_{i}}, u_{g_{j}})$, $a_{ijk} = ord_{x}lcm(u_{g_{i}}, u_{g_{j}}, u_{g_{k}})$,\, $b_{ijk} = ord_{\partial}lcm(u_{g_{i}}, u_{g_{j}}, u_{g_{k}})$,\, $\dots $\,\,($1\leq i, j, k,\dots \leq d$). Then $U''_{rs} = \bigcup_{i=1}^{d}\{[\Theta(r - a_{i}, s - b_{i}) \setminus
\Theta(r - a_{i}, s - c_{i})]u_{g_{i}}\}$. By the combinatorial principle of inclusion and exclusion (see [5, Chapter 5, Theorem 5.1.1]),

$Card\,U''_{rs} = \D\sum_{i=1}^{d}Card\,\{[\Theta(r - a_{i}, s - b_{i})\setminus \Theta(r - a_{i}, s - c_{i})]u_{g_{i}}\} - $

\noindent$\D\sum_{1\leq i<j\leq d}Card\,\{[\Theta(r - a_{i}, s - b_{i}) \setminus\Theta(r - a_{i}, s - c_{i})]u_{g_{i}}\bigcap [\Theta(r - a_{j}, s - b_{j}) \setminus $
 
\noindent$\Theta(r - a_{j}, s - c_{j})]u_{g_{j}}\} + \D\sum_{1\leq i<j<k\leq d}Card\,\{[\Theta(r - a_{i}, s - b_{i})\setminus\Theta(r - a_{i}, s$

\medskip

\noindent$- c_{i})]u_{g_{i}}\bigcap[\Theta(r - a_{j}, s - b_{j})\setminus \Theta(r - a_{j}, s - c_{j})]u_{g_{j}} \bigcap [\Theta(r - a_{k}, s - b_{k})$

\medskip

\noindent$\setminus \Theta(r - a_{k}, s - c_{k})]u_{g_{k}}\} - \dots $.

\medskip

Furthermore, for any two different elements $g_{i}, g_{j}\in \Sigma$, we have 

\medskip

\noindent$Card\,\{[\Theta(r - a_{i}, s - b_{i})
\setminus \Theta(r - a_{i}, s - c_{i})]u_{g_{i}}\bigcap [\Theta(r - a_{j}, s - b_{j}) \setminus \Theta(r - a_{j},$

\medskip

\noindent$ s - c_{j})]u_{g_{j}}\} =
Card\,\{\theta\,lcm(u_{g_{i}}, u_{g_{j}}) | \theta \in \Theta, ord_{x}\theta \leq r-a_{ij},\, ord_{\partial}\theta \leq s-b_{ij},$

\medskip

\noindent$ord_{\partial}(\theta \frac{lcm(u_{g_{i}}, u_{g_{j}})}{u_{g_{i}}}v_{g_{i}}) = ord_{\partial}\theta + b_{ij} - b_{i} + c_{i} > s$ and $ord_{\partial}\theta + b_{ij} - b_{i} + c_{i} > s\}$
 
\medskip
 
\noindent$= Card\,\{\theta | \theta\in \Theta, ord_{x}\theta \leq r-a_{ij}, ord_{\partial}\theta \leq s-b_{ij}$\,\,  and 

\medskip

\noindent$ord_{\partial}\theta > s - \min\{c_{i}+b_{ij} - a_{i}, c_{j}+b_{ij} - a_{j}\} =$
 
\medskip
 
\noindent$\D{r + n - a_{ij}\choose m}\left[{s + n - b_{ij}\choose n} - \D{s + n - \min \{c_{i}+b_{ij} - b_{i}, c_{j}+b_{ij} - b_{j} \}\choose n}\right]$

\medskip

Similarly, for any three different elements $g_{i}, g_{j}, g_{k}\in \Sigma$ we obtain that 

\medskip

$Card\,\{[\Theta(r - a_{i}, s - b_{i}) \setminus \Theta(r - a_{i}, s - c_{i})]u_{g_{i}}\bigcap[\Theta(r - a_{j}, s - b_{j}) \setminus \Theta(r - a_{j}, s - c_{j})]u_{g_{j}}\bigcap[\Theta(r - a_{k}, s - b_{k}) \setminus \Theta(r - a_{k}, s - c_{k})]u_{g_{k}}\} = \D{r + n - a_{ijk}\choose n}[{s + n - b_{ijk}\choose n}$

\noindent$- \D{s + n - \min \{c_{i}+b_{ijk} - b_{i}, c_{j}+b_{ijk} - b_{j}, c_{k}+b_{ijk} - b_{k} \}\choose n}]$ and so on.

Thus, for all sufficiently large  $(r,s)\in {\bf N}^2$, $Card\,U''_{rs} = \bar{\omega}(r, s)$ where $\bar{\omega}(t_{1}, t_{2})$ is the following
numerical polynomial:
$$\bar{\omega}(t_{1}, t_{2}) =
\sum_{i=1}^{d}{t_{1}+n - a_{i}\choose n}\left[{t_{2} + n - b_{i}\choose n} -
{t_{2} +n - c_{i}\choose n}\right] -$$
$$\sum_{1\leq i < j\leq d}
{t_{1} + n - a_{ij}\choose n}\left[{t_{2} + n - b_{ij}\choose n}-
{t_{2} + n - \min \{c_{i}+b_{ij} - b_{i}, c_{j} + b_{ij} - b_{j} \}
\choose n}\right]$$
$$+ \sum_{1\leq i < j < k\leq d}
{t_{1} + n - a_{ijk}\choose n}[{t_{2} +n - b_{ijk}\choose n} -$$
\begin{equation}{t_{2} +n - \min \{c_{i}+b_{ijk} - b_{i}, c_{j}+b_{ijk} - b_{j},
c_{k}+b_{ijk} - b_{k}\}\choose n}] - \dots 
\end{equation}
It is clear now that the
numerical polynomial $\phi_{M}(t_{1},t_{2}) = \omega(t_{1}, t_{2}) +
\bar{\omega}(t_{1}, t_{2})$ has all the desired properties.
\end{proof}
\begin{definition}
Numerical polynomial $\phi_{M}(t_{1}, t_{2})$, whose existence is established by
Theorem {\em 5.4} , is called a characteristic polynomial of the
module $M$ associated with the system of generators $\{f_{1},\dots, f_{m}\}$
(or with the bifiltration $(M_{rs})_{r,s\in {\bf N}}$.)
\end{definition}
\begin{example}
{\em With the notation of Theorem 5.4, let $n=1$ and let an $A_{1}(K)$-module $M$ be generated by a single element $f$ that satisfies the defining
equation\, $x^{2}f + \partial^{2}f + x\partial f = 0$\,.  In other words, $M$ is a factor module of a free $A_{1}(K)$-module $E = A_{1}(K)e$ with a free generator $e$ by its $A_{1}(K)$-submodule $N = A_{1}(K)g$ where $g = (x^{2} + \partial^{2} + x\partial)e$. Clearly, $\{g\}$ is an 
$(x, \partial)$-Gr\"obner basis of $N$. Applying Proposition 3.5 (and using the notation of Theorem 5.4), we obtain that $u_{g} = x^{2}e$,
$v_{g}=\partial^{2}e$, and $$\omega(t_{1}, t_{2}) =
\omega_{(2, 0)}(t_{1}, t_{2}) = {{t_{1}+1}\choose 1}{{t_{2}+1}\choose 1} -
{{t_{1}+1-2}\choose 1}{{t_{2}+1}\choose 1} =2t_{2} + 2.$$ Furthermore,
formula (5.14) shows that $$\bar{\omega}(t_{1}, t_{2}) =
{{t_{1}+1-2}\choose 1}\left[{{t_{2}+1}\choose 1} - {{t_{2}+1-2}\choose 1}\right] =
2t_{1} - 2.$$ Thus, the characteristic polynomial of the module $M$
associated with the generator $f$ is as follows:
$$\phi_{M}(t_{1}, t_{2}) = \omega(t_{1}, t_{2}) + \bar{\omega}(t_{1}, t_{2})
= 2t_{1} + 2t_{2}.$$}
\end{example}
\begin{theorem} 
Let $M$ be a finitely generated $A_{n}(K)$-module and let $$\phi_{M}(t_{1}, t_{2}) = \sum_{i=0}^{n}\sum_{j=0}^{n} {a_{ij}{t_{1}+i\choose i}{t_{2}+j\choose j}}$$ be a characteristic polynomial associated with some finite system of generators $\{g_{1}, \dots, g_{p}\}$
of $M$. (We write $\phi_{M}(t_{1}, t_{2})$ in the form {\em (3.1)} with integer coefficients $a_{ij}, 1\leq i, j\leq n$.) Furthermore, let $\Lambda = \{(i, j)\in {\bf N}^2\,|\,0\leq i, j\leq n$ and $a_{ij}\neq 0\}$, and let $\mu=(\mu_{1}, \mu_{2})$ and $\nu=(\nu_{1}, \nu_{2})$ be the maximal
elements of the set $\Lambda$ relative to the lexicographic and reverse lexicographic orders on ${\bf N}^2$, respectively.
Then $d = deg\,\phi_{M}$, $a_{nn}$, $\mu$, $\nu$, the coefficients $a_{mn}$, $a_{\mu_{1}, \mu_{2}}$ $a_{\nu_{1}, \nu_{2}}$ of the polynomial  $\phi_{M}(t_{1}, t_{2})$, and the coefficients of all terms of $\phi_{M}(t_{1}, t_{2})$of total degree $d$  do not depend on the finite system
of generators of the $A_{n}(K)$-module $M$ this polynomial is associated with.
\end{theorem}
\begin{proof}
Let $\{h_{1}, \dots, h_{q}\}$ be another finite system of generators of the $A_{n}(K)$-module $M$ and let $(M_{rs})_{r, s\in {\bf Z}}$ and
$(M'_{rs})_{r, s\in {\bf Z}}$ be the the natural bifiltrations associated with the systems of generators $\{g_{1}, \dots, g_{p}\}$ and
$\{h_{1}, \dots, h_{q}\}$, respectively. ($M_{rs} = \D\sum_{i=1}^{p}W_{rs}g_{i}$ and $M'_{rs} = \D\sum_{i=1}^{q}W_{rs}h_{i}$ for $r, s\in {\bf N}$,
$M_{rs} = 0$ and $M'_{rs} = 0$, if at least one of the indices $r, s$ is negative.) Furthermore, let $\phi^{\ast}_{M}(t_{1}, t_{2}) = \D\sum_{i=0}^{n}\D\sum_{j=0}^{n}{b_{ij}{t_{1}+i\choose i}{t_{2}+j\choose j}}$ ($b_{ij}\in {\bf Z}$ for $i=0,\dots, n; j= 0,\dots, n$) be the characteristic polynomial associated with the system $\{h_{1}, \dots, h_{q}\}$, let $\Lambda' = \{(i, j)\in {\bf N}^2\,|\,0\leq i, j\leq n$, and $b_{ij}\neq 0\}$, and  let $\sigma = (\sigma_{1}, \sigma_{2})$, $\epsilon =(\epsilon_{1}, \epsilon_{2})$ be the maximal elements of $\Lambda'$ relative to the lexicographic and reverse lexicographic orders on ${\bf N}^2$, respectively. In order to prove the theorem, we should show that $\mu = \sigma$, $\nu = \epsilon$, $a_{nn} = b_{nn}$, $a_{\mu_{1}\mu_{2}} = b_{\sigma_{1}\sigma_{2}}$, and $a_{\nu_{1}\nu_{2}} = b_{\epsilon_{1}\epsilon_{2}}$.

Since $\bigcup_{r, s\in {\bf {Z}}}M_{rs} = \bigcup_{r, s\in {\bf {Z}}}M'_{rs} = M$, there exist elements $r_{0}, s_{0}\in {\bf N}$ such that
$M_{rs}\subseteq M'_{r+r_{0}, s+s_{0}}$ and $M'_{rs}\subseteq M_{r+r_{0}, s+s_{0}}$ for all $r, s\in {\bf N}$. It follows that $\phi_{M}(r, s)\leq \phi^{\ast}_{M}(r+r_{0}, s+s_{0})$ and $\phi^{\ast}_{M}(r, s)\leq \phi(r+r_{0}, s+s_{0})$ for all sufficiently large $r, s\in {\bf N}$, say for all $r\geq r_{1}, s\geq s_{1}$ where $r_{1}$ and $s_{1}$ are some positive integers. Therefore, $deg\,\phi_{M} = deg\,\phi^{\ast}_{M}$, 
$a_{nn} = (n!)^{2}\D\lim_{r\rightarrow \infty, s\rightarrow \infty} \frac{\phi_{M}(r, s)}{r^{m}s^{n}}\leq (n!)^{2}\D\lim_{r\rightarrow \infty, s\rightarrow \infty}\D\frac{\phi^{\ast}_{M}(r+r_{0}, s+s_{0})}{r^{n}s^{n}} = (n!)^{2}\D\lim_{r\rightarrow \infty, s\rightarrow \infty}
\D\frac{\phi^{\ast}_{M}(r, s)}{r^{n}s^{n}} = b_{nn}$
and similarly $b_{nn}\leq a_{nn}$, so that $b_{nn} = a_{nn}$.

\medskip

If $a_{nn}\neq 0$, then $(n, n)\in \Lambda$ and $(n, n)\in \Lambda'$ hence $\mu = \nu = \sigma = \epsilon = (n, n)$ and
$a_{\mu_{1}\mu_{2}} = a_{\nu_{1}\nu_{2}} = b_{\sigma_{1}\sigma_{2}} = b_{\epsilon_{1}\epsilon_{2}} = a_{nn} = b_{nn}$. \, Suppose that
$a_{nn} = 0$. Then $(\mu_{1}, \mu_{2})\neq (n, n)$, $a_{\mu_{1}\mu_{2}}\neq 0$, and the coefficient of the monomial
$t_{1}^{\mu_{1}}t_{2}^{\mu_{2}}$ in the polynomial $\phi_{M}(t_{1}, t_{2})$ is equal to $\D\frac{a_{\mu_{1}\mu_{2}}}{\mu_{1}!\mu_{2}!}$. \,

\medskip

Let $s\in {\bf N}$, $s\geq s_{1}$, and let $e$ be a positive integer such that $s^{e}\geq r_{1}$. By the choice of the elements $\mu$ and $\sigma$,
$\phi_{M}(s^{e}, s) = \D\frac{a_{\mu_{1}\mu_{2}}}{\mu_{1}!\mu_{2}!}s^{e\mu_{1}+\mu_{2}} + o(s^{e\mu_{1}+\mu_{2}})$ and 
$\phi^{\ast}_{M}(s^{e}, s) = \D\frac{b_{\sigma_{1}\sigma_{2}}}{\sigma_{1}!\sigma_{2}!}s^{e\sigma_{1}+\sigma_{2}} + o(s^{e\sigma_{1}+\sigma_{2}})$
for all sufficiently large values of $e$.\,
Since $\phi_{M}(s^{e}, s)\leq \phi^{\ast}(s^{e}+p, s+q) = \D\frac{b_{\sigma_{1}\sigma_{2}}}{\sigma_{1}!\sigma_{2}!}s^{e\sigma_{1}+\sigma_{2}} + o(s^{e\sigma_{1}+\sigma_{2}})$ and
$\phi^{\ast}_{M}(s^{e}, s)\leq \phi(s^{e}+p, s+q) = \D\frac{a_{\mu_{1}\mu_{2}}}{\mu_{1}!\mu_{2}!}s^{e\mu_{1}+\mu_{2}} + o(s^{e\mu_{1}+\mu_{2}})$
for all $s\geq s_{0}$, we conclude that $e\mu_{1} + \mu_{2} = e\sigma_{1} + \sigma_{2}$ for all sufficiently large $e\in {\bf N}$ and the coefficients of the power $s^{e\mu_{1} + \mu_{2}}$ in the polynomials $\phi_{M}(t_{1}, t_{2})$ and $\phi^{\ast}_{M}(t_{1}, t_{2})$ are equal. Therefore,
$\mu_{1} = \sigma_{1}$, $\mu_{2} = \sigma_{2}$ and $a_{\mu_{1}\mu_{2}}= b_{\mu{1}\mu_{2}}$. \, The equalities  $\nu_{1} = \epsilon_{1}$,
$\nu_{2} = \epsilon_{2}$ and $a_{\nu_{1}\nu_{2}}= b_{\nu_{1}\nu_{2}}$, as well as the equality of coefficients of total degree $d$ of  can be proved similarly.
\end{proof}
It is clear that if $(W_{r})_{r\in {\bf Z}}$ is the one-dimensional
filtration of the Weyl algebra $A_{n}(K)$ introduced in Section 2, then
$W_{r}\subseteq D_{rr}\subseteq W_{2r}$ for all $r\in {\bf N}$.
Therefore, if $\phi_{M}(t_{1}, t_{2})$ and $\psi_{M}(t)$ denote,
respectively, our characteristic polynomial and the Bernstein polynomial
associated with the same finite system of generators of an
$A_{n}(M)$-module $M$, then
$\psi_{M}(r)\leq \phi_{M}(r, r)\leq \psi_{M}(2r)$ for all sufficiently
large $r\in {\bf Z}$. It follows that
$n\leq deg\,\psi_{M}(t) = deg\,\phi_{M}(t_{1}, t_{2})\leq 2n$ and
$M\in \mathcal{B}_{n}$ if and only if $deg\,\phi_{M}(t_{1}, t_{2}) = n$.

\medskip

The following example shows that a characteristic polynomial
$\phi_{M}(t_{1}, t_{2})$ of a finitely generated $A_{n}(K)$-module
$M$ can carry more invariants (i. e., numbers that do not depend on the
choice of a system of generators the characteristic polynomial is
associated with) than the Bernstein polynomial $\psi_{M}(t)$.

\begin{example}
{\em With the notation of Theorem 5.4, let an $A_{1}(K)$-module
$M$ be generated by a single element $f$ that satisfies the defining
equation\,
\begin{equation}
x^{a}\partial^{b}f + \partial^{a+b}f = 0 
\end{equation}
where $a$ and $b$ are some positive integers. In other words, $M$ is a factor module of a free $A_{1}(K)$-module $E = A_{1}(K)e$ with a free generator
$e$ by its $A_{1}(K)$-submodule $N = A_{1}(K)g$ where $g = (x^{a} + \partial^{2} + x\partial)e$. Obviously, $\{g\}$ is an $(x, \partial)$-Gr\"obner basis of the module $N$. Since $u_{g} = x^{a}\partial^{b}e$ and  $v_{g}=\partial^{a+b}e$, we obtain (using the notation of Theorem 5.4) that $\omega(t_{1}, t_{2}) = \omega_{\{(a, b)\}}(t_{1}, t_{2}) = \D{{t_{1}+1}\choose 1}{{t_{2}+1}\choose 1} -
\D{{t_{1}+1-a}\choose 1}\D{{t_{2}+1-b}\choose 1} =bt_{1} + at_{2} + a + b - ab$.
Furthermore, formula (5.14) shows that $\bar{\omega}(t_{1}, t_{2}) =
\D{{t_{1}+1-a}\choose 1}[\D{{t_{2}+1-b}\choose 1} - \D{{t_{2}+1-(a+b)}\choose 1}] =
at_{1} + a(1-a)$. Thus, the characteristic polynomial of the module $M$
associated with the generator $f$ is as follows:
$\phi_{M}(t_{1}, t_{2}) = \omega(t_{1}, t_{2}) + \bar{\omega}(t_{1}, t_{2})
= (a+b)t_{1} + at_{2} + 2a + b - ab - a^{2}$.

The Bernstein polynomial $\psi_{M}(t)$ associated with the
generator $f$ of the
$A_{1}(K)$-module $M$ can be obtained from the exact sequence of
finitely generated filtered modules
\[ 0 \longrightarrow
F^{a+b}\xrightarrow{\text{$\alpha$}}E\xrightarrow{\text{$\pi$}}
M\longrightarrow 0 \]
where $M$ and $E$ are equipped, respectively, with the filtrations
$(W_{r}e)_{r\in {\bf Z}}$ and $(W_{r}f)_{r\in {\bf Z}}$ defined in
Section 2, and $F^{a+b}$ is a free filtered $A_{1}(K)$-module with a
single free generator $h$ and filtration $(W_{r-(a+b)}h)_{r\in {\bf Z}}$.
(Here $\pi$ denotes the natural $A_{1}(K)$-epimorphism of $E$ onto $M$ that
maps $e$ onto $f$, and $\alpha$ is the natural $A_{1}(K)$-epimorphism
of the free filtered module $F^{a+b}$ onto the $A_{1}(K)$-module
$N\subseteq M$ equipped with the filtration $(W_{r}g)_{r\in {\bf Z}}$, \,
$\alpha(h) = g$.)

Since $dim_{K}W_{r} = Card\,\{x^{i}\partial^{j} | i + j\leq r\}
= \D{r+2\choose 2}$ for all $r\in {\bf N}$,

\noindent$dim_{K}(W_{r-(a+b)}h) = \D{r+2-(a+b)\choose 2}$ for all sufficiently
large $r\in {\bf Z}$ whence

$\psi_{M}(r) = dim_{K}(W_{r}e) - dim_{K}(W_{r-(a+b)}h) =
\D{r+2\choose2}-\D{r+2-(a+b)\choose2}$ for all
sufficiently large $r\in {\bf Z}$. Therefore,
$\psi_{M}(t) =
\D{t+2\choose2}-\D{t+2-(a+b)\choose2} = (a+b)t - \D\frac{(a+b)(a+b-3)}{2}\,.$

Comparing the polynomials $\psi_{M}(t)$ and
$\phi_{M}(t_{1}, t_{2})$
we see that the first polynomial carries two
invariants, its degree 1 and the leading coefficient $a+b$, while
$\phi_{M}(t_{1}, t_{2})$ carries three such
invariants, its total degree 1, $a+b$, and $a$.
Thus, $\phi_{M}(t_{1}, t_{2})$
gives both parameters $a$ and $b$ of the equation (5.2) while
the Bernstein polynomial $\psi_{M}(t)$ gives just the sum of the parameters.}
\end{example}

\section*{Acknowledgements}

The first author's research was partially supported by the Austrian Science Fund (FWF): W1214-N15, project DK11 and project no. P20336-N18 (DIFFOP) as well as the Austrian Marshall Plan Foundation: scholarship no. 256 420 24 7 2011.

\medskip

The second author's research was supported by the NSF Grant CCF 1016608

\bigskip

\centerline{\large R\normalsize EFERENCES}

\medskip

\noindent[{\bf 1}] \, T. Becker  and V. Weispfenning,
{\em Gr\"obner Bases. A Computational Approach to Commutative Algebra\/}
(Springer-Verlag, Berlin, Heidelberg, New York, 1993).

\noindent[{\bf 2}] \, I. N. Bernstein, `Modules over the ring of
differential operators. A study of the fundamental solutions
of equations with constant coefficients',
{\em Funct. Anal. and its Appl.\/} Vol. 5 (1971), 89 - 101.

\noindent[{\bf 3}] \, I. N. Bernstein, `The analytic continuation of
generalized functions with respect to a parameter',
{\em Funct. Anal. and its Appl.\/} Vol. 6 (1972), 273 - 285.

\noindent[{\bf 4}] \, J.- E. Bj\"ork, {\em Rings of Differential Operators\/}
(North Holland Publishing Company, Amsterdam, New York,  1979).

\noindent[{\bf 5}] \, P. J. Cameron,
{\em Combinatorics. Topics, Techniques, Algorithms\/}
(Cambridge University Press, Cambridge, 1994).

\noindent[{\bf 6}] \, D. Eisenbud, {\em Commutative Algebra with a View
Toward Algebraic Geometry\/} (Springer-Verlag, Berlin, Heidelberg,
New York, 1995).

\noindent[{\bf 7}] \, M. Insa and F.Pauer, `Gr\"obner Bases in Rings of
Differential Operators', {\em Gr\"obner Bases and Applications\/}
(Cambridge Univ. Press, New York, 1998), 367 - 380.

\noindent[{\bf 8}] \, E. R. Kolchin, {\em Differential Algebra and Algebraic
Groups\/} (Academic Press, New York, 1973).

\noindent[{\bf 9}] \, M. V. Kondrateva, A. B. Levin, A. V. Mikhalev, and
E. V. Pankratev,
`Computation of Dimension Polynomials', {\em Internat. J. Algebra
and Comput\/} Vol. 2 (1992), 117 - 137.

\noindent[{\bf 10}] \, M. V. Kondrateva, A. B. Levin, A. V. Mikhalev, and
E. V. Pankratev,
{\em Differential and Difference Dimension Polynomials\/}
(Kluwer Academic Publishers, Dordrecht, Boston, London, 1999).

\end{document}